\documentclass[]{amsart}

\usepackage{amscd,amsthm,amssymb,amsfonts,amsmath,euscript}
\usepackage[english]{babel}
\usepackage{xcolor}
\usepackage{comment}
\usepackage[all]{xy}
\usepackage[margin=1.25in]{geometry}
\usepackage{hyperref}

\newtheorem{thm}{Theorem}[section]

\newtheorem{lemma}[thm]{Lemma}

\newtheorem{proposition}[thm]{Proposition}

\theoremstyle{definition}
\newtheorem{remark}[thm]{Remark}

\newtheorem{example}[thm]{Example}

\numberwithin{equation}{section}
\newcommand{\nc}{\newcommand}

\def\makeop#1{\expandafter\def\csname#1\endcsname
  {\mathop{\rm #1}\nolimits}\ignorespaces}
\makeop{Hom}   \makeop{End}   \makeop{Aut}   \makeop{Isom}  \makeop{Pic} 
\makeop{Gal}   \makeop{ord}   \makeop{Char}  \makeop{Div}   \makeop{Lie} 
\makeop{PGL}   \makeop{Corr}  \makeop{PSL}   \makeop{sgn}   \makeop{Spf}
\makeop{Spec}  \makeop{Tr}    \makeop{Nr}    \makeop{Fr}    \makeop{disc}
\makeop{Proj}  \makeop{supp}  \makeop{ker}   \makeop{im}    \makeop{dom}
\makeop{coker} \makeop{Stab}  \makeop{SO}    \makeop{SL}    \makeop{SL}
\makeop{Cl}    \makeop{cond}  \makeop{Br}    \makeop{inv}   \makeop{rank}
\makeop{id}    \makeop{Fil}   \makeop{Frac}  \makeop{GL}    \makeop{SU}
\makeop{Nrd}   \makeop{Sp}    \makeop{Tr}    \makeop{Trd}   \makeop{diag}
\makeop{Res}   \makeop{ind}   \makeop{depth} \makeop{Tr}    \makeop{st}
\makeop{Ad}    \makeop{Int}   \makeop{tr}    \makeop{Sym}   \makeop{can}
\makeop{length}\makeop{SO}    \makeop{torsion} \makeop{GSp} \makeop{Ker}
\makeop{Adm}   \makeop{Mat}
\def\makebb#1{\expandafter\def
  \csname bb#1\endcsname{{\mathbb{#1}}}\ignorespaces}
\def\makebf#1{\expandafter\def\csname bf#1\endcsname{{\bf
      #1}}\ignorespaces} 
\def\makegr#1{\expandafter\def
  \csname gr#1\endcsname{{\mathfrak{#1}}}\ignorespaces}
\def\makescr#1{\expandafter\def
  \csname scr#1\endcsname{{\EuScript{#1}}}\ignorespaces}
\def\makecal#1{\expandafter\def\csname cal#1\endcsname{{\mathcal
      #1}}\ignorespaces} 

\def\doLetters#1{#1A #1B #1C #1D #1E #1F #1G #1H #1I #1J #1K #1L #1M
                 #1N #1O #1P #1Q #1R #1S #1T #1U #1V #1W #1X #1Y #1Z}
\def\doletters#1{#1a #1b #1c #1d #1e #1f #1g #1h #1i #1j #1k #1l #1m
                 #1n #1o #1p #1q #1r #1s #1t #1u #1v #1w #1x #1y #1z}
\doLetters\makebb   \doLetters\makecal  \doLetters\makebf
\doLetters\makescr 
\doletters\makebf   \doLetters\makegr   \doletters\makegr

    \def\setminus{\smallsetminus}

\normalsize

\makeop{Bl}

\def\Fq{{\bbF}_q}

\newcommand{\Z}{\mathbb Z}
\newcommand{\Q}{\mathbb Q}
\newcommand{\R}{\mathbb R}

\newcommand{\A}{\mathbb A}    
\newcommand{\F}{\mathbb F}

\newcommand{\pr}{\indent }
\newcommand{\co}{\mathcal{O}}

\newcommand{\<}{\langle}   
\renewcommand{\>}{\rangle} 

\nc{\embed}{\hookrightarrow}

\nc{\ol}{\overline}
\nc{\wt}{\widetilde}
\nc{\opp}{\mathrm{opp}}

\def\wh{\widehat}
\def\pr{\mathrm{pr}}

\newcommand{\fct}[4]{\begin{split}#1&\longrightarrow#2\\#3&\longmapsto#4\end{split}}

\makeop{Ram}
\makeop{Rep}

\begin{document}
\title{The Chevalley--Gras formula over
        global fields}


\author{Jianing Li}
\address{Jianing Li\\CAS Wu Wen-Tsun Key Laboratory of Mathematics\\ University of Science and Technology of China \\ Hefei, Anhui 230026, China}
\email{lijn@ustc.edu.cn}
\urladdr{http://staff.ustc.edu.cn/$\sim$lijn}

\author{Chia-Fu Yu}
\address{Chia-Fu Yu\\
        Institute of Mathematics, Academia Sinica and NCTS \\
        Astronomy Mathematics Building \\
        No. 1, Roosevelt Rd. Sec. 4 \\ 
        Taipei, Taiwan, 10617} 
\email{chiafu@math.sinica.edu.tw}



\maketitle


\begin{abstract}
  In this article we give an adelic proof of the Chevalley--Gras formula
  for global fields, which itself is a generalization of  
  the ambiguous class number formula. The idea is to reduce the
  formula to the Hasse norm theorem and to the local and
  global reciprocity laws. We also give an adelic proof of the
  Chevalley--Gras formula for the   class group of divisors of degree $0$ in the
  function field case, which extends a result of Rosen.
\end{abstract} 

\maketitle

\bigskip
\section{Introduction}
\label{sec:01}

Let $K/k$ be a cyclic extension of number fields with Galois group
$G$. Let $\grm$ be a modulus of $k$, which gives rise to a
modulus $\grm_K$ of $K$.
The ray class group $\Cl^{\grm_K}_K$ modulo $\grm_K$ 
admits a $G$-module structure.
The Chevalley--Gras formula describes an explicit relationship 
between the generalized ambiguous ray class number 
$|(\Cl^{\grm_K}_{K}/{\mathcal{C}})^G|$  and 
$|\Cl^\grm_{k}/{N(\mathcal{C})}|$, where $\mathcal{C}\subset
\Cl^{\grm_K}_{K}$ is any $G$-submodule and $N$ is the norm map from $K$ to
$k$. In the case when the submodule
$\mathcal{C}$ and the modulus $\grm$ are trivial, 
the formula then relates the class numbers 
$|\Cl_K^G|$ and $|\Cl_k|$, which is  the classical ambiguous class
number formula due to Chevalley \cite[p.~406]{chevalley}. A proof of
Chevalley's formula 
can be found in Gras' book \cite[Lemma 6.1.2 and Remark 6.2.3]{gras:cft} or
in Lang's book \cite[Chapter 13, \S4, Lemma 4.1]{lang:cyc}. Lemmermeyer \cite{lemmermeyer:amb}
gives an elementary proof which follows closely the approach taken by  
Lang, but avoiding the machinery of cohomologies. 
The existing proofs of Chevalley's formula are reduced to a result
of the  Herbrand quotient of global units.  
In \cite[Th\'{e}or\`{e}me 4.3]{gras:73}, 
Gras gave a formula for narrow class groups with arbitrary
$\mathcal{C}$.   
In \cite[Th\'{e}or\`{e}me 2.7]{gras:94} 
(also see the English translation \cite[Section 2]{gras:17}), 
he proved this formula for ray class groups.  
His proof is based on Chevalley's formula. Recently, 
a generalization of Chevalley's class number formula to dihedral extensions
has been investigated by Caputo and Nuccio \cite{caputo-nuccio}.
 
In this article we give an adelic proof of the Chevalley--Gras formula
over global fields. More precisely, using the adelic language, we
reduce the formula 
to the Hasse norm theorem and to the
local and global index theorems, which is shorter and more conceptual.

In the function field case, the class group of divisors of degree $0$  
deserves a special attention. 
The ambiguous class number formula (the case $\mathcal{C}=0$) 
for functions fields 
was obtained by Rosen \cite
{rosen:amb-funfield}. We also give an adelic proof of the formula with
an arbitrary $G$-submodule $\mathcal{C}$.

In the last section we add an elementary exposition of a cohomological
variant for $S$-ray  class groups, for the sake of
completeness. This formulation is valid for an arbitrary Galois
extension $K/k$, and is essentially equivalent to Chevalley's original
formula in the cyclic case, 
thanks to the theorem on the Herbrand quotient of
global units.

\section{The Chevalley--Gras formula}
\label{sec:CG}

In this section, we  recall the definition of the $S$-ray class group
and prove the Chevalley--Gras formula. We then give some special cases
of this formula for future convenience. In Example~\ref{eg: redei},
we  use this formula  to reprove a classical result of R\'{e}dei on the
$4$-rank of the narrow class group of quadratic fields as this
approach does not seem to appear in the literature.

\subsection{Notation and $S$-ray class groups}
\label{sec:CG.1}

Let $F$ be a global field, that is, $F$ is either a number field 
(a finite extension of $\Q$) or a global function field 
(a finite extension of $\F_p(t)$ for some prime $p$).
Let $V_F$ denote the set of all places of $F$, and let $V_{F,\infty}$
(resp. $V_{F,f}$) denote the subset of archimedean (resp.~finite)
places. (So in the function field case, $V_{F,\infty}=\emptyset$.)  For each place $w\in V_F$, 
the completion of $F$ at $w$ is denoted by
$F_w$. 
Let $\calO_w$ denote the ring of integers of $F_w$ if $w$ is
finite. The canonical embedding from $F$ to the completion $F_w$ is 
also denoted by $w$.

The letter $S$ always denotes a non-empty finite set 
of places of $F$ containing $V_{F,\infty}$.  Denote by $\calO_{F,S}$
the ring of $S$-integers of $F$, which consists of all
elements $a\in F$ such that $a \in \co_w$ for all $w\not\in S$. 
An $S$-modulus $\grm$ is a formal product $\grm_\infty\cdot \grm_f$,
where 
$\grm_f$ is a nonzero integral ideal of $\calO_{F,S}$ and 
$\grm_\infty$ is a formal
product of some  real places if $F$ is a number field and 
$\grm_\infty$ is always $1$ otherwise.
Let $S(\grm):=\{w\in V_F|  w|\grm\}$ be the support of $\grm$.
Let $I_F$ be the free abelian group
generated by $V_{F,f}$, and $I_F^{S(\grm)}$ the subgroup generated by
$V_{F,f}\setminus S(\grm)$. The ideal $\grm_f$ corresponds to an effective
divisor in $I_F$ whose support is disjoint from $S$ (in $V_{F,f}$).
Let \[ F^\grm:=\{x\in F^\times| x\equiv 1\bmod \grm_f \text{ and
  } w(x)>0 \text{ for each real place } w \mid \grm_{\infty}\}. \]
Let $i:F^\times \to I_{F}$ be the natural map defined by $a\mapsto
\sum_{w\in V_{F,f}} \ord_w(a) w$. Note that
 $i(F^\grm)\subset I^{S(\grm)}_F$.  
The \emph{ray class group  of $F$ modulo  $\grm$}  is defined as 
 \[ \Cl^{\grm}_{F} :=  I^{S(\grm)}_F/{i(F^{\grm})}.  \]
The \emph{$S$-ray class group of $F$ modulo $\grm$} is defined as 
\begin{equation}
  \label{eq:CG.1}
  \Cl^{\grm}_{F,S} := \Cl^\grm_{F}/\langle{\text{image of $S\cap
     V_{F,f}$}}\rangle.
\end{equation}
Since $S$ is non-empty, $\Cl^{\grm}_{F,S}$ is finite. 

Alternatively, let $I_{F}^{S(\grm) \cup S}\subset I^{S(\grm)}_F$ be the
subgroup generated by $V_{F,f}\setminus (S\cup S(\grm))$. Then we have a
projection $\pr: I^{S(\grm)}_F \to I^{S(\grm)}_F/\<S\cap V_{F,f}\>\cong
I_{F}^{S(\grm) \cup S}$. Composing with $i$, we obtain a map $i_S:F^\grm
\to I_{F}^{S(\grm) \cup S}$, which maps $a$ to 
$\mathrm{div}^S(a):=\sum_{w\not\in S} \ord_w(a) w$. 
Then one can define $\Cl^{\grm}_{F,S}$ by  \[I_{F}^{S(\grm) \cup
  S}/i_S(F^\grm),\] and this agrees with the definition
\eqref{eq:CG.1}. The group $I_{F}^{S}$ can also be naturally identified with 
the ideal group of $\calO_{F,S}$.
Under this identification, the map
$i_S: F^\grm \to I_{F}^{S(\grm) \cup S}$ sends $a$ to the principal
ideal $a \calO_{F,S}$. Put $P_{F}^{\grm,S}:=i_S(F^\grm)$, the subgroup
of principal  $S$-ideals modulo $\grm$,  and then we have 
$\Cl^{\grm}_{F,S}=I_{F}^{S(\grm) \cup S}/P_{F}^{\grm,S}$. 
In the case where 
$\grm=1$, this is the \emph{$S$-ideal class group of $F$} and  
is denoted by $\Cl_{F,S}$.

For convenience, we also let $\calO_w:=F_w$ if $w\in V_{F,\infty}$,
and for each $w|\grm$ we also write 
\[ 
1+\grm \calO_w:=
\begin{cases}
  1+\grm_f \calO_w & \text{if $w\mid \grm_f$;} \\
  (F_w^\times)^2 & \text{if $w \mid \grm_\infty$.} \\
\end{cases} \]

Let $K/k$ be a finite Galois extension of global fields with Galois group $G$. Let
$(S_K,\grm_K)$ be a pair consisting of a finite set $S_K$ of places 
and a modulus
$\grm_K$ of $K$ as above. Suppose that both $S_K$ and $\grm_K$ are
$G$-invariant. Then the $S_K$-ray class group $\Cl^{\grm_K}_{K,S_K}$
admits an action of $G$. So for any $G$-submodule $\calC$ of 
$\Cl^{\grm_K}_{K,S_K}$, one may look for 
a relationship between $|(\Cl^{\grm_K}_{K,S_K}/\mathcal{C})^G|$ and
 $|\Cl^{\grm}_{k,S}/N(\mathcal{C})|$, for a suitable pair $(S,\grm)$
 for $k$ related to $(S_K,\grm_K)$, where $N$ is the norm map from
 $K$ to $k$. This question is  answered mostly when $K/k$ is cyclic
 and remains open in general, even for the abelian case.

Suppose $(S,\grm)$ is a pair for $k$. Let $S_K$ be the set of places
of $K$ over $S$ and let $\grm_K:=\grm_{K,\infty} \cdot \grm_{K,f}$,
where $\grm_{K,\infty}$ is the set of real places of $K$ over the
support of $\grm_\infty$ and $ \grm_{K,f}:=\grm_f \calO_{K,S}$. Then
$(S_K,\grm_K)$ is a $G$-invariant  pair, and we say $(S_K,\grm_K)$ is
induced by $(S,\grm)$. In this case, we also write
$\Cl^\grm_{K,S}$ for $\Cl^{\grm_K}_{K,S_K}$ and call it the $S$-ray
class group of $K$ modulo $\grm$. When $\grm=1$, we write also $\Cl_{K,S}$ 
for $\Cl_{K,S_K}$ and call it the $S$-ideal
class group of $K$.

\subsection{The main formula}
\label{sec:CG.2}
Let $K/k$ be a cyclic extension of global fields with group
$G=\<\sigma\>$, where $\sigma$ is a generator. Let $N=N_{K/k}$ be 
the norm map from $K$ to $k$. 
Let $S\supset V_{k,\infty}$ be a finite non-empty set of places of $k$
and $\grm$ an $S$-modulus. Let $\Cl^\grm_{K,S}:=\Cl^{\grm_K}_{K,S_K}$
be  the $S$-ray class group modulo $\grm$, where $(S_K,\grm_K)$ is
the pair induced by $(S,\grm)$.

For $v\in V_{k,f}$, denote by $e_v$ and $f_v$ 
the ramification index and inertia
degree of $v$ in $K/k$ respectively.  In the number field case, if $v$ is real and every place
$w|v$ of $K$ is complex, we say that $v$ is ramified in $K$, 
and put $e_v=2$ and $f_v=1$,  otherwise, we put $e_v=f_v=1$. 
The following theorem which we call the Chevalley--Gras formula over
global fields is proved
by Gras in the number field case; see \cite[Theorem 3.6]{gras:17}.

 \begin{thm}\label{CG}
Let $K/k$ be a cyclic extension of global fields with Galois group
$G$. Let $\grm$ be a modulus of $k$, and let $S\supset V_{k,\infty}$ 
be a finite non-empty set of places of $k$ such that $S\cap
S(\grm_f)=\emptyset$.
Let $\mathcal{C}$ be a 
$G$-submodule of the $S$-ray class group  $\Cl^\grm_{K,S}$. 
Let    $D$ be any subgroup of $I^{S(\grm)}_{K}$ such that the image of $D$ in $\Cl^\grm_{K,S}$ is equal to
$\mathcal{C}$.
Then  
\[
\frac{|(\Cl^\grm_{K,S}/\mathcal{C})^G|}{|\Cl^\grm_{k,S}/{N(\mathcal{C})}|}
= \frac{\prod \limits_{v \in S\setminus S(\grm)}  e_v f_v
  \prod\limits_{v\in S(\grm_f)} [1+\grm\co_v: N(\prod\limits_{w\mid
    v}(1+\grm\co_w))] \prod\limits_{v\notin S\cup S(\grm)}e_v
  }{[K:k][\Lambda:\Lambda\cap N(K^\grm)]}. \] 
Here $\Lambda=\{x \in k^\grm | (x)\co_{k,S}= N(d)\calO_{k,S}   \text{ in }
I^{S}_k  \text{ for some } d\in D \}$.
 \end{thm}

\begin{remark}\label{rmk: after thm}
\begin{enumerate}
\item  The group $\Lambda$ depends on the choice of $D$,
  however, we will see in \eqref{eq:indD} that the index $[\Lambda:\Lambda\cap
  N(K^\grm)]$ depends only on $\calC$. 
\item The group $N(\prod\limits_{w\mid v}(1+\grm\co_w))$  
equals  $ N_{K_w/k_v}(1+\grm\co_w)$  for any $w\mid v$ as $K/k$ is Galois. 
\end{enumerate}
\end{remark}

\begin{proof}
        We first express $\Cl^\grm_{K,S}/\mathcal{C}$ in terms of
        ideles. 
Let $\A^\times_K$ denote the idele group of $K$.  
Let $\A^\grm_K=\{(a_w)_w \in \A^\times_K| a_w \in 1+\grm\co_w
\text{ for each } w\mid \grm \}$. We denote the  canonical surjection
$\A^\grm_K  \rightarrow I^{S(\grm)}_K$ by $\pi$. The kernel of $\pi$  is $
U^\grm_{K}:= \prod_{w\nmid \grm}\co^\times_w \prod_{w\mid \grm}
1+\grm\co_w .$  
Put 
\[U^{\grm}_{K,S}=U^{\grm}_K \prod_{w\in S_K,
  w\nmid \grm_\infty} K^\times_w=\prod_{w\nmid \grm, w\not\in S_K} 
\co^\times_w \prod_{w\mid \grm}
1+\grm\co_w \prod_{w\in S_K,
  w\nmid \grm_\infty} K^\times_w.\] 
Then $\pi$ induces  an
isomorphism   \[ \A^\grm_K/{K^{\grm}
  U^{\grm}_{K,S}}\cong\Cl^\grm_{K,S}.\] 
By the approximation theorem,   $\A^\grm_K K^\times = \A^\times_K$. Note that $\A^\grm_K\cap K^\times U^{\grm}_{K,S} = K^\grm U^{\grm}_{K,S}$.   So the inclusion  $\A^\grm_K\subset \A^\times_K$ induces an isomorphism 
\[  \Cl^\grm_{K,S} \cong \A^\times_K/K^\times U^\grm_{K,S}.   \]
Put $\widetilde{D}=\pi^{-1}(D) K^\times U^\grm_{K,S}$. It follows that 
\[  \Cl^\grm_{K,S}/\mathcal{C} \cong\A^\times_K/\widetilde{D}.\]

Since $H^1(G,\A^\times_K)=0$, the fact $G=\langle \sigma \rangle$ is
cyclic implies that  $ (\A^\times_K)^{1-\sigma}$ 
is the kernel of the norm from $\A^\times_K$ to $\A^\times_k$.
 So there is an exact commutative diagram  
\[
\begin{CD}
1 @>>> \widetilde{D} \cap  (\A^\times_K)^{1-\sigma}  @>>> \widetilde{D}   @>N >>
N(\widetilde{D} ) @>>> 1 \\
@. @VVV @VVV @VVV \\
1 @>>>  (\A^\times_K)^{1-\sigma} @>>> \A_K^\times @>{N}>>
N(\A^\times_K) @>>> 1.  \\    
\end{CD}
\] 
The snake lemma gives  the short exact sequence
\[ 1 \rightarrow (\Cl^\grm_{K,S}/{\mathcal{C}})^{1-\sigma}\rightarrow
\Cl^\grm_{K,S}/{\mathcal{C}} \rightarrow
N(\A^\times_K)/N(\widetilde{D}) \rightarrow 1,\] as one has $
(\A^\times_K)^{1-\sigma}/ (\A^\times_K)^{1-\sigma}\cap
\widetilde{D})\cong  (\Cl^\grm_{K,S}/{\mathcal{C}})^{1-\sigma}$.  For
any finite $G$-module $M$, one has $|M^{G}|=|M/{M^{1-\sigma}}|$ 
by the exact sequence  
\[ 0 \rightarrow M^G \rightarrow M \xrightarrow[]{1-\sigma}  M
\rightarrow M/{M^{1-\sigma}}\rightarrow 0. \] 
Thus we obtain the equality   \[|(\Cl^\grm_{K,S}/{\mathcal{C}})^G|=|
N(\A^\times_K)/N(\widetilde{D})|.\]  
Recall that Hasse's norm theorem  says that $k^\times \cap
N(\A^\times_K)=N(K^\times)$.  
So given an element  $N(x)=a N(d)\in  N(\A^\times_K) \cap k^\times
N(\widetilde{D})$ with $x\in \A^\times_K$, $a\in k^\times $ and $d\in
\widetilde{D}$,  we have  $a=N(y)$ for some $y\in K^\times$. Hence $
N(\A^\times_K)\cap k^\times N(\widetilde{D}) = N(K^\times
)N(\widetilde{D}) = N(\widetilde{D})$.  
Therefore the natural map  
\[N(\A^\times_K)/N(\widetilde{D}) \rightarrow  k^\times
N(\A^\times_K)/k^\times N(\widetilde{D}) \] 
is an isomorphism.  The global index theorem \cite[Chapter IX,
\S5]{lang:ant} says that $|\A^\times_k/{k^\times
  N(\A^\times_K)}|=|G|=[K:k]$. This implies that $k^\times
N(\A^\times_K)/k^\times N(\widetilde{D})$ is a 
subgroup of $\A^\times_k/{k^\times N(\widetilde{D})}$ with index
$[K:k]$. Therefore   
\[ |(\Cl^\grm_{K,S}/{\mathcal{C}})^G|=[K:k]^{-1}|\A^\times_k/{k^\times
  N(\wt D)}|.\]
To compute $\A^\times_k/{k^\times N(\widetilde{D})}$, we consider the
exact sequence 
\begin{equation}
  \label{eq:CG.4}
1 \rightarrow  k^\times N(\widetilde{D})U^\grm_{k,S}   /k^\times N(\widetilde{D})  \rightarrow \A^\times_k/k^\times N(\widetilde{D}) \rightarrow \A^\times_k/k^\times U^\grm_{k,S}  N(\widetilde{D})   \rightarrow 1.  
\end{equation}
We claim that the last term is isomorphic to $\Cl^\grm_{k,S}/N(\mathcal{C})$.   By the identification $\A^{\grm}_k/k^\grm U^\grm_{k,S}\cong \Cl^\grm_{k,S}$,  we have that  $N(\mathcal{C}) \subset \Cl^\grm_{k,S}$ is the image of $N(\pi^{-1}(D)) $ in $\Cl^\grm_{k,S}$. Hence 
$$\Cl^\grm_{k,S}/N(\mathcal{C})\cong \A^{\grm}_k/N(\pi^{-1}(D))k^\grm U^\grm_{k,S}.$$
Then the inclusion $\A^\grm_k\hookrightarrow \A^\times_k$ induces an isomorphism 
\[\A^{\grm}_k/N(\pi^{-1}(D))k^\grm U^\grm_{k,S} \cong \A^{\times}_k/N(\pi^{-1}(D))k^\times U^\grm_{k,S}= \A^\times_k/k^\times U^\grm_{k,S}  N(\widetilde{D}).  \]

The first term of \eqref{eq:CG.4}  can be computed by the exact sequence
\begin{equation}
  \label{eq:CG.5}
  1 \rightarrow U^\grm_{k,S}\cap k^\times N(\widetilde{D}) /N(U^\grm_{K,S})  \rightarrow U^\grm_{k,S}/N(U^\grm_{K,S}) \rightarrow  k^\times N(\widetilde{D})U^\grm_{k,S}   /k^\times N(\widetilde{D}) \rightarrow 1.
\end{equation}

Let  $G_v$ and 
$I_v$ be  the decomposition group  and inertia group of $v$
respectively.  For each place $v$ of $k$, we choose a place $w$  of
$K$ above $v$.  By local class field theory,  $H^2(G_v,
K_w^\times)=k^\times_v/{N_{K_w/{k_v}}(K^\times_w)}\cong G_v$ and
$H^2(G_v,\co^\times_w)=\co^\times_v/{N_{K_w/k_v}(\co^\times_w)}\cong
I_v$. It follows from the
cyclicity of $G$ and Shapiro's Lemma that \[\begin{split} 
U^\grm_{k,S}/N(U^\grm_{K,S}) 
& \cong \prod_{v\in S\setminus S(\grm)} H^2(G_v, K^\times_w) \times  \prod_{v\in S(\grm)}  \frac{(1+\grm\co_v)}
{N_{K_w/{k_v}}(1+\grm\co_w)}\times  \prod_{v\notin S(\grm)\cup S} H^2(G_v, \co^\times_w)   \\&
\cong \prod_{v\in S\setminus S(\grm)}
G_v \times \prod_{v\in S(\grm_f)} \frac{(1+\grm\co_v)}
{N_{K_w/{k_v}}(1+\grm\co_w)} \times
\prod_{v\notin S\cup S(\grm)} I_v. \\
\end{split} \] 
Here we use our condition that $S$ is disjoint with $S(\grm_f)$ and the fact that $1+\grm\co_v=\R_{>0}=N_{K_w/k_v}(1+\grm\co_w)$ if $v\mid \grm_\infty$. This contributes to the numerator of the right hand side
term in the theorem.  
In order to prove the theorem, it suffices to show that the first
term of \eqref{eq:CG.5} is isomorphic to
$\Lambda/\Lambda\cap N(K^\grm)$.

Recall that $\pi$ is the natural projection $\A^\grm_{K}\rightarrow I^{S(\grm)}_K$. 
Write $\bar{D}=\pi^{-1}(D)$ for simplicity.  
As $U^\grm_{k,S}N(\bar{D})\subset \A^\grm_k$, it is direct to check that 
 \begin{equation}\label{eq: CG5.5}
\Lambda=k^\grm  \cap U^\grm_{k,S}N(\bar{D})=  k^\times  \cap U^\grm_{k,S}N(\bar{D}).
 \end{equation}  Given $x=u N(\bar{d})\in \Lambda$ with $u\in U^\grm_{k,S}$ and $\bar{d}\in \bar{D}$, we define  a function $f$ as follows:
\[ \fct{f: \Lambda}{U^\grm_{k,S}\cap k^\times N(\bar{D}U^\grm_{K,S})/N(U^\grm_{K,S}) }{x}{u\bmod N(U^\grm_{K,S})}.\]
We need to  show that $f$ is well-defined. Suppose $x=uN(\bar{d})=u'N(\bar{d'}) \in \Lambda$ with $u,u'\in U^\grm_{k,S}$ and $\bar{d},\bar{d'}\in \bar{D}$. Then $u'/u=N(\bar{d}/\bar{d'})\in N(\bar{D})\cap U^\grm_{k,S}\subset N(\A^\grm_K)\cap U^\grm_{k,S}$. By Lemma~\ref{lem: idelic-compute}(1), the last group coincides with $N(U^\grm_{K,S})$ . So $f$ is a well-defined map. 

It is clear that $f$ is a group homomorphism.  We  show that $f$ is
surjective.  Let $u= t N(\bar{d} a)$ be an element of
$U^\grm_{k,S}\cap k^\times N(\bar{D}U^\grm_{K,S})$ with $t\in
k^\times,  \bar{d}\in \bar{D}$ and $a\in U^\grm_{K,S}$.  Then
$t=uN(a)^{-1} N(\bar{d})^{-1}$ with  $uN(a)^{-1}\in U^\grm_{k,S}$ and
$N(\bar{d})^{-1}\in N(\bar{D})$.  Note that $t$ is in fact in
$k^\grm$. This shows $t\in \Lambda$ by \eqref{eq: CG5.5}.  We have
$f(t)=uN(a)^{-1} \bmod N(U^\grm_{K,S}) \equiv u \bmod
N(U^\grm_{K,S})$. This proves the surjectivity.

 The kernel of $f$ by definition coincides with $\Lambda \cap N(U^\grm_{K,S}\bar{D})$.  Lemma~\ref{lem: idelic-compute}(3) shows that it also equals  $\Lambda\cap N(K^\grm)$. Thus,  as desired,   $f$ induces an isomorphism
\begin{equation}
   \label{eq:indD}
   \Lambda/\Lambda\cap N(K^\grm) \cong U^\grm_{k,S}\cap k^\times N(\bar{D}U^\grm_{K,S})/N(U^\grm_{K,S}).
\end{equation}
Observe that the term $k^\times N(\bar{D}U^\grm_{K,S})$ is independent
of the choice of $D$ as $k^\times N(K^\times
\bar{D}U^\grm_{K,S})=k^\times N(\wt D)$.  
This finishes the proof of the theorem.   \end{proof}

\begin{lemma}\label{lem: idelic-compute}
We have the following equalities:
\begin{enumerate}
\item $N(\A^\grm_K)\cap U^\grm_{k,S}=N(U^\grm_{K,S})$;

\item $N(K^\times)\cap N(\A^\grm_K)=N(K^\grm)$;

\item $\Lambda \cap N(U^\grm_{K,S}\bar{D})=\Lambda\cap
N(K^\grm)=\Lambda\cap N(\A^\grm_K)$. 
\end{enumerate}
\end{lemma}

\begin{proof} Recall $\A^\grm_K=\{(a_w)_w \in \A^\times_K| a_w \in 1+\grm\co_w
        \text{ for each } w\mid \grm \}$.  As mentioned in Remark~\ref{rmk: after thm},  
$N(\prod_{w\mid v}K^\times_w)=N_{K_w/k_v}(K^\times_w)$ for each $w$ as
$K/k$ is Galois.  For a place $w\nmid \grm$, it is easy to see that $N_{K_w/k_v}(K^\times_w)\cap
\co^\times_v=N_{K_w/k_v}(\co^\times_w)$. We obtain $N(\prod_{w\mid
  v}K^\times_w)\cap \co^\times_v =N(\prod_{w\mid v}\co^\times_w)$. For a place $w\mid \grm$, we have the trivial equality $N_{K_w/k_v}(1+\grm\co_w)  \cap 1+\grm\co_v = N_{K_w/k_v}(1+\grm\co_w)$. It follows that   $N(\A^\grm_K)\cap U^\grm_{k,S}=N(U^\grm_{K,S})$. This proves (1).

 To prove  (2), let $K^\times \times \A^\grm_K $ denote the direct product of $K^\times$ and $\A^\grm_K$. 
 Consider the exact commutative diagram 
\[
\begin{CD}
1 @>>> K^\times\cap \A^\grm_K    @>{x\mapsto (x,x)}>> K^\times \times \A^\grm_K   @>{(a,b)\mapsto ab^{-1}} >>
K^\times \A^\grm_K @>>> 1 \\
@. @V N VV @V N\times N VV @V N VV \\
1 @>>> N(K^\times)\cap N(\A^\grm_K)  @>{x\mapsto (x,x)}>> N(K^\times)\times N(\A^\grm_K)  @>{(a,b)\mapsto ab^{-1}}>>
N(K^\times) N(\A^\grm_K)  @>>> 1.  \\    
\end{CD}
\] 
Note that $K^\times\cap \A^\grm_K$ by definition is $K^\grm$. By the
approximation theorem, $K^\times \A^\grm_K=\A^\times_K$. Since
$H^1(G,\A^\times_K)=H^1(G,K^\times)=0$ and $G$ is cyclic,  the snake
lemma gives an exact sequence 
\[ (K^\times)^{1-\sigma}\times (\A^\grm_K\cap
(\A^\times_K)^{1-\sigma})\rightarrow
(\A^\times_K)^{1-\sigma}\rightarrow  N(K^\times)\cap
N(\A^\grm_K)/{N(K^\grm)}\rightarrow 0.\] 
The first arrow is surjective by the weak approximation theorem. Thus the
last term is $0$. This proves (2). 

(3) Hasse's norm theorem says that $k^\times \cap
    N(\A^\times_K)=N(K^\times)$. Recall that $\Lambda=k^\times \cap
    U^\grm_{k,S}N(\bar{D})$ by \eqref{eq: CG5.5}.  By (2), we
    have 
\[\Lambda\cap N(U^\grm_{K,S}\bar{D})  = k^\times \cap
    N(U^\grm_{K,S}\bar{D})\subset N(K^\times)\cap
    N(U^\grm_{K,S}\bar{D})\subset N(K^\times)\cap
    N(\A^\grm_K)=N(K^\grm). \] 
This proves the inclusion $\Lambda\cap
    N(U^\grm_{K,S}\bar{D})\subset \Lambda \cap N(K^\grm)$. To show the
    other inclusion,  note that $U^\grm_{k,S}N(\bar{D})\cap
    N(\A^\grm_K)=N(U^\grm_{K,S}\bar{D})$ by (1).  Then  
\[\Lambda\cap N(K^\grm)\subset \Lambda\cap N(\A^\grm_K)=k^\times \cap
    N(U^\grm_{K,S}\bar{D})=\Lambda \cap N(U^\grm_{K,S}\bar{D}).\]  
The last equality follows from \eqref{eq: CG5.5}.

The second equality in (3) follows from 
\[\Lambda \cap N(\A^\grm_K) \subset \Lambda \cap N(K^\times)\cap N(\A^\grm_K) = \Lambda \cap N(K^\grm). \] This completes the proof of the lemma.   \end{proof}

\begin{remark}
  The idea of our adelic proof of Theorem~\ref{CG} comes from 
  \cite{yu:amb}, which shows that Chevalley's ambiguous class number
  formula follows immediately from the Hasse norm theorem, and the
  local and global norm index theorems. When the extension $K/k$ is
  abelian, we know that the analogous statements for the 
  local and global norm index theorems hold
  true; see \cite[Chapter~IX, Sections 3 and 5]{lang:ant}. However, to
  extend Chevalley's formula to abelian extensions, the
  assumption that $K/k$ is cyclic is crucial in the argument used 
  in \cite{yu:amb}.  
\end{remark}

\subsection{Examples}
\label{sec:CG.3}
We list some  special cases of the Chevalley--Gras formula in the number field case. 
\begin{example}
        $(1)$ If $\grm=1$ and $S$ is the set of infinite places,  then
        $\Cl^\grm_{K,S}$ is equal to  $\Cl_K$,  the class group of $K$. The theorem says 
        \[  \frac{|(\Cl_K/\mathcal{C})^G|}{|\Cl_k/{N(\mathcal{C})}|}=\frac{\prod_{ v\leq \infty}e_v}{[K:k][\Lambda:\Lambda\cap N(K^\times)]}.
        \]
        If we let $\mathcal{C}$ and $D$ be trivial, then
        $\Lambda=\co^\times_k$, the unit group of $\co_k$. The formula
        becomes the ambiguous class number formula for the class group 
        \[  \frac{|(\Cl_K)^G|}{|\Cl_k|}=\frac{\prod_{v\leq
        \infty}e_v}{[K:k][\co^\times_k :\co^\times_k \cap
        N(K^\times)]}. 
        \]

        $(2)$ If  $\grm$ is the product of all the real places of $k$ and $S$ is the set of infinite places, then $\Cl^\grm_{K,S}$ is the narrow class group  $\Cl^{+}_K$ of $K$.  Similarly, $\Cl^\grm_{k,S}=\Cl^{+}_k$. Note that $K^\grm$ is equal to  $K^{+}$,  the group of totally positive elements of $K^\times$.  The theorem says \begin{equation}\label{eq:gras}
\frac{|(\Cl^{+}_K/\mathcal{C})^G|}{|\Cl^{+}_k/{N(\mathcal{C})}|}=\frac{\prod_{v\nmid \infty}e_v}{[K:k][\Lambda:\Lambda\cap N(K^+)]}.  
        \end{equation} 
                If we further  let $\mathcal{C}$ and $D$ be trivial, then $\Lambda=(\co^\times_k)^{+}$. The formula becomes the ambiguous class number formula for narrow class groups which was first proved by Chevalley in \cite[p.~406]{chevalley}. 
        \end{example}

We now use  the formula \eqref{eq:gras} to reprove 
a classical result of R\'{e}dei.

\begin{example}[$4$-rank of  narrow class groups of quadratic fields]
  \label{eg: redei} 
   Let $K$ be a quadratic number field with discriminant $d$. Let
           $T=\{p_1,\cdots,p_t\}$ be the set of prime numbers ramified
           in $K$.  Let $G=\Gal(K/\Q)=\langle \sigma \rangle$. For
           $a\in \Cl^{+}_K$, $N_{K/\Q}(a)=a a^\sigma=1$ as $\Q$ has
           class number $1$. This implies that  
           $|\Cl^{+}_K[2]|:=|\{a\in \Cl^{+}_K|
           a^2=1\}|=|(\Cl^{+}_K)^G|$. The latter term has cardinality
           $2^{t-1}$ by Chevalley's formula \eqref{eq:gras}. In other
           words, the $2$-rank of $\Cl^{+}_K$ is $t-1$. 
        The following $\F_2$-matrix is the R\'{e}dei matrix:
        \begin{equation}\label{eq:redeimatrix}
R:=\left(\log(p_i,d)_{p_j}\right)_{1\leq i,j\leq t}.
        \end{equation} 
        Here $\log: \{\pm1\}\twoheadrightarrow \F_2$ is the logarithm
        map and $(p_i,d)_{p_j}$ is the quadratic Hilbert
        symbol of $p_i$ and $d$  at  the prime $p_j$. Note that the sum of each row of this matrix
        is zero by the product formula of Hilbert symbols. A theorem of R\'{e}dei  \cite[Theorem 3.1]{stevenhagen:redei} says that
          the $4$-rank of $\Cl^{+}_K$ is $t-1-\mathrm{rank}(R)$.
          
          The matrix $R$ is the transpose of the matrix in
          \cite[Theorem 3.1]{stevenhagen:redei}.  One can check
          that the logarithm Hilbert symbol  $\log(p_i,d)_{p_j}$
          coincides with the logarithm Kronecker symbol
          $\left(\frac{p^{*}_j}{p_i}\right) \in \F_2$  when $i\neq j$. Here
          $p^*=(-1)^{{p-1}/2}p$ for $p$ is odd. If $2\mid d$, $2^*$ is
          the number such that $d=\prod_{p\mid d}p^{*}$.
\end{example}

The proof in \cite{stevenhagen:redei} uses the explicit construction
of the $2$-Hilbert class field. We give a proof by applying
\eqref{eq:gras} to $K/\Q$.  
By definition, the $4$-rank of $\Cl^{+}_K$ is $\mathrm{rank}_{4}
\Cl^{+}_K=\dim_{\F_2}\Cl^{+}_K[4]/\Cl^{+}_K[2].$   As we mentioned,
$a^{\sigma}=a^{-1}$ for $a\in \Cl^{+}_K$.  It follows that    
\[a \bmod \Cl^{+}_K[2] \in (\Cl^{+}_K/{\Cl^{+}_K[2]})^G
\Leftrightarrow a^\sigma a^{-1} = a^{-2} \in \Cl^{+}_K[2]
\Leftrightarrow a\in \Cl^{+}_K[4].\] 
This shows $\Cl^{+}_K[4]/\Cl^{+}_K[2] = (\Cl^{+}_K/{\Cl^{+}_K[2]})^G$. We now use \eqref{eq:gras} to compute the order of this group.

Take $\mathcal{C}=\Cl^{+}_K[2]$ in  \eqref{eq:gras}. 
  It is well known that $\mathcal{C}$ is generated by the ramified
  prime ideals.  We add a proof for this fact here for
  the sake of completeness. Suppose $I\in I_K$ such that its image
  $\mathrm{cl}(I)$ is in $\Cl^{+}_K[2]=(\Cl^{+}(K))^G$. Then $I^\sigma
  I^{-1}$  is generated by  some totally positive element $\alpha\in
  K^\times$.  Since $N(I^\sigma I^{-1})=(1)$ in $I_\Q$, we have
  $N(\alpha)=\pm 1$.  Thus $N(\alpha)=1$ by the positivity of $\alpha$.  By
  Hilbert's Theorem 90, $\alpha =\beta^\sigma \beta^{-1}$ for some
  $\beta\in 
  K^\times$. Note that we can assume $\beta$ is totally positive. Thus
  $I\beta^{-1}$ is a $G$-invariant fractional ideal of $K$.  It
  follows that $I\beta^{-1}\in \langle D, I_\Q\rangle$ (see
  Lemma~\ref{A.lem}(1)), where $D\subset I_K$ is a subgroup such that
  ${\rm cl}(D)=\calC=\Cl_K^+[2]$ in $\Cl^+_K$.  
  Thus $\mathrm{cl}(I)\in \mathrm{cl}(D)$.   
        
Let $D$ be the subgroup of $I_K$  generated by the ramified prime
ideals.  We have shown that $D$ generates $\mathcal{C}$.   The group
$\Lambda$ of \eqref{eq:gras} is  then the subgroup of $\Q^\times$
generated by $T=\{p_1,\cdots, p_t\}$. Consider the
following map 
\[  \Lambda\rightarrow  \prod_{j=1}^{t}\{\pm  1\}, \quad  x \mapsto
\left((x,d)_{p_j}\right)_j. \] 
The kernel is $\Lambda\cap N(K^\times)$ by the properties of Hilbert
symbols and Hasse's norm theorem,  for details see \cite[Lemma
2.8]{ljn:clgp}.    
  Note that $\Lambda \cap N(K^\times) = \Lambda \cap N(K^+)$ as $\Lambda \subset \Q^{+}$.  
 The image has size $2^r$ where $r=\mathrm{rank}(R)$
is the rank of the R\'{e}dei matrix \eqref{eq:redeimatrix}. Therefore 
        \[|\Cl^{+}_K[4]/\Cl^{+}_K[2]|=|(\Cl^{+}_K/{\Cl^{+}_K[2]})^G| =2^{t-1-r}.\]

\section{The case of  class groups of divisors of degree $0$}
We  let $K/k$ be a cyclic extension of global function fields 
with Galois group $G$. Denote by $\F_{q'}$ and $\F_q$  the
constant fields of $K$ and $k$, respectively. 
Let $\A^0_K$ be the kernel of the degree map
 \[\deg_K:\A^\times_K \rightarrow \Z, \quad (x_w)_w \mapsto \sum_{w}
{\ord_w(x_w)[k_w:\F_{q'}]},\] where $k_w$ is the residue field of $w$.  Let $U_K=\prod_{w}\co^\times_w$.  
The  class group of divisors  and the  class group of divisors of degree $0$ of $K$  are
defined  respectively by  
\[   \Cl_K=\A^\times_K/{U_K K^\times} \quad \text{ and } \quad
\Cl^0_K=\A^0_K/{U_K K^\times}.\] 
It is well known that $\Cl^0_K$  is finite. 
The degree map induces the
exact sequence 
\[ 0\rightarrow \Cl^0_K \rightarrow \Cl_K \xrightarrow[]{\deg_K}  \Z
\rightarrow 0.  \] 
See \cite[Chapter V, Theorem 5]{artin-tate} 
for the surjectivity of $\deg_K$. We define $\Cl_k$, $\Cl^0_k, \deg_k, U_k$ for $k$ in the same way. Let $N$ denote the norm map from $K$ to $k$.  For a prime divisor $w\in \A^\times_K/{U_K}$ of $K$, by definition $N(w)=v^{[k_w:k_v]}$ where $v$ is the prime divisor of $k$ below $w$. This implies  $\deg_k(N(\A^\times_K))=[\F_{q'}:\F_q]\Z$.

Let $\mathcal{C}$ be a $G$-submodule of $\Cl^{0}_K$. 
Choose any subgroup $D$ of $\A^0_K$ such that the image of $D$ 
in $\Cl^0_K$ is equal to $\mathcal{C}$, and put 
$\Lambda:=k^\times \cap N(D)U_k$ in
$\A^\times_k$. Note that $\Lambda$ depends on the choice of $D$,
however, its image in $k^\times\cap N(D)N(K^\times) U_k/N(K^\times)$
depends only on $\mathcal{C}$. In particular, 
the index $[\Lambda:\Lambda\cap N(K^\times)]$ depends only on $\mathcal{C}$.
Let $d(K/k)\in \Z$ denote the positive generator of the ideal
$\deg_K(\Cl^G_K)$ of $\Z$. 

\begin{thm}\label{thm: function_field} With notation as above, 
        one has  \[
        |(\Cl^0_K/{\mathcal{C}})^G|=|\Cl^0_k/{N(\mathcal{C})}|
        \frac{[\F_{q'}:\F_q]\prod_v e_v}{[K:k][\Lambda:\Lambda\cap
        N(K^\times)]}d(K/k). \] 
\end{thm}

\begin{remark} 
    Putting $\mathcal{C}=0$ and $D=0$, 
    we obtain the following formula 
    \[|(\Cl^0_K)^G|=|\Cl^0_k|
        \frac{[\, \F_{q'}:\F_q]\prod_v e_v}{[K:k]\cdot [\,\Fq^\times:\Fq^\times \cap
        N(K^\times)]}d(K/k). \]
    When $q'=q$, this recovers the ambiguous
    class number formula obtained by Rosen (see \cite[Theorem
    8 and Proposition 2]{rosen:amb-funfield}).
     It is shown in \cite[p.164]{rosen:amb-funfield} 
     that the invariant $d(K/k)$ divides another invariant
     $\delta(K/k)$ which is easier to compute.
     Rosen also computed  $d(K/k)$ in some special cases; see
        \cite[Theorem 4]{rosen:amb-funfield}. For example, if the
        cyclic extension $K/k$ is unramified everywhere, then
        $d(K/k)=[K:k]$; see \cite[Corollary to Theorem 4]{rosen:amb-funfield}. 
\end{remark}

\begin{lemma}
Let $\sigma$ be a generator of $G$. For any  $G$-submodule
$\mathcal{C}\subset \Cl^0_K$, we have \[
d(K/k)=|(\Cl_K/{\mathcal{C}})^{1-\sigma}/
(\Cl^0_K/{\mathcal{C}})^{1-\sigma}|. \] 
\end{lemma}
\begin{proof}
        This follows from the exact sequences 
        \[ \begin{xymatrix}{
                & 0 \ar@{->}[d] & 0 \ar@{->}[d]& 0 \ar@{->}[d] & \\ 
                0  \ar@{->}[r] & (\Cl^0_K/{\mathcal{C}})^G
                \ar@{->}[r] \ar@{->}[d] & (\Cl_K/{\mathcal{C}})^G
                \ar@{->}[r]\ar@{->}[d] & d(K/k)\Z \ar@{->}[r]
                \ar@{->}[d] & 0\\ 
                0 \ar@{->}[r]  & \Cl^0_K/{\mathcal{C}} \ar@{->}[r]
                \ar@{->}[d] & \Cl_K/{\mathcal{C}}
                \ar@{->}^-{\deg_K}[r] \ar@{->}^-{1-\sigma}[d] & \Z
                \ar@{->}[r] \ar@{->}[d] & 0     \\       
                0 \ar@{->}[r]  & (\Cl^0_K/{\mathcal{C}})^{1-\sigma}
                \ar@{->}[r] \ar@{->}[d] &
                (\Cl_K/{\mathcal{C}})^{1-\sigma}
                \ar@{->}[r]\ar@{->}[d] & \Z/{d(K/k)\Z}
                \ar@{->}[r]\ar@{->}[d] & 0 \\ 
                & 0 & 0 & 0 & \text{ }
        }  \end{xymatrix}  \] \end{proof}

Now we give an adelic proof of Theorem \ref{thm: function_field}. The
reader will realize  that  the proof is analogous to that   of
Theorem~\ref{CG}. 
\begin{proof}
        Put $\widetilde{D}=D K^\times U_K$. 
        The facts   $H^{1}(G,\A^\times_K)=\hat{H}^{-1}(G,\A^\times_K)=0$ and
        $(\A^\times_K)^{1-\sigma}\subset \A^0_K$ give the  exact
        commutative diagram         
        \[
        \begin{CD}
        1 @>>> \widetilde{D} \cap  (\A^\times_K)^{1-\sigma}  @>>>
        \widetilde{D}   @>N >> 
        N(\widetilde{D} ) @>>> 1 \\
        @. @VVV @VVV @VVV \\
        1 @>>>  (\A^\times_K)^{1-\sigma} @>>> \A^{0}_K @>{N}>>
        N(\A^0_K) @>>> 1.  \\    
        \end{CD}
        \]
As  $(\A^\times_K)^{1-\sigma}/\widetilde{D} \cap
        (\A^\times_K)^{1-\sigma} \cong
        (\Cl_K/\mathcal{C})^{1-\sigma}$,   
        the snake lemma gives the short exact sequence
        \[0 \rightarrow (\Cl_K/{\mathcal{C}})^{1-\sigma}\rightarrow
        \Cl^0_K/{\mathcal{C}} \rightarrow N(\A^0_K)/N(\widetilde{D})
        \rightarrow 0.\]  
        We remark that $(\Cl_K/\mathcal{C})^{1-\sigma}$ is finite
        although $\Cl_K/\mathcal{C}$ is infinite.  By the above lemma, 
        \[ 
        \begin{split}
        |(\Cl^0_K/{\mathcal{C}})^G|&= 
        |N(\A^0_K)/N(\widetilde{D})|\cdot
         |(\Cl_K/{\mathcal{C}})^{1-\sigma}/
         (\Cl^0_K/{\mathcal{C}})^{1-\sigma}| \\
        & =d(K/k)|N(\A^0_K)/N(\widetilde{D})| \\ 
        & = d(K/k)|N(\A^0_K) k^\times/N(\widetilde{D} ) k^\times|.         
        \end{split}
        \] 
       We prove the last equality as follows.   Let $N(x)=N(d)a\in
       N(\A^0_K)\cap N(\widetilde{D})k^\times$ with $x\in \A^0_K,
       d\in \widetilde{D}$ and $a \in k^\times$. Then $a=N(xd^{-1})\in
       k^\times \cap N(\A^0_K)\subset k^\times \cap
       N(\A^\times_K)=k^\times \cap N(K^\times)$ by Hasse's norm
       theorem.  Then the inclusion $N(\A^0_K)\subset
       N(\A^0_K)k^\times$ induces an  isomorphism  
       \[   N(\A^0_K)/N(\widetilde{D}) \cong N(\A^0_K) k^\times/N(\widetilde{D} ) k^\times.\]
    Consider the  short exact sequence
        \[0\rightarrow N(\A^0_K) k^\times/N(\widetilde{D}) k^\times
         \rightarrow \A^0_k/N(\widetilde{D}) k^\times  \rightarrow
         \A^0_k/N(\A^0_K)k^\times\rightarrow 0.\] 
        Suppose that $k=\F_q(E)$ is the function field of  some curve
         $E$. We apply the degree map to the Artin reciprocity map
         ${\rm Art}:\A^\times_k/{k^\times N(\A^\times_K)}\cong G$ and
         obtain the 
         following exact commutative
         diagram 
         \begin{equation}\label{Art}
 \begin{xymatrix}{
                0  \ar@{->}[r] &\A^0_k/{k^\times N(\A^0_K)}
                \ar@{->}[r] \ar@{->}[d]^-{\cong}_{\varphi_1} &
                \A^\times_k/{k^\times N(\A^\times_K)}
                \ar@{->}^-{\deg_k}[r]\ar@{->}[d]^-{\cong}_{\rm Art} &
                \Z/{[\F_{q'}:\F_q]\Z} \ar@{->}[r]
                \ar@{->}[d]^-{\cong}_{\varphi_2} & 0  \\ 
                0  \ar@{->}[r] &\Gal(K/\F_{q'}(E))   \ar@{->}[r]  & G
                \ar@{->}[r] & \Gal(\F_{q'}/\F_q) \ar@{->}[r]  & 0.  
        }\end{xymatrix}          
         \end{equation}
Note that the isomorphism $\varphi_2$ is induced by the Frobenius
map. The commutativity of the right diagram follows from 
\cite[Chapter VIII, Theorem 10]{artin-tate}. 
By the Corollary of \cite[Chapter VIII, Theorem 10]{artin-tate}, 
the map $\varphi_1$ is surjective and hence is an isomorphism.

It follows from \eqref{Art} that  
\[|\A^0_k/{k^\times N(\A^0_K)}
                |=\frac{[K:k]}{[\F_{q'}:\F_q]}.\] 
To prove the theorem, it remains to show that 
        \[|\A^0_k/N(\widetilde{D})
                k^\times|=|\Cl^0_k/{N(\mathcal{C})}| \frac{\prod_v
                e_v}{[\Lambda: \Lambda\cap N(K^\times)]}.\] 
The proof is the same as that following equation \eqref{eq:CG.4}
                in the proof of Theorem~\ref{CG}, and is omitted.  
\end{proof}

\section{A cohomological variant for $S$-ray class groups}\label{sec:A}

Let $K/k$ be a finite Galois extension of global fields with Galois group
$G$. As in Section~\ref{sec:CG.1}, we 
let $\grm$ be a modulus of $k$, and 
$S$ be a non-empty finite set of places of $k$ containing 
all archimedean places which is disjoint from the support of
$\grm_f$. 
In this section we shall discuss a cohomological variant of the 
ambiguous $S$-ray class number formula of $K/k$; 
see Theorem~\ref{A.1}. This formulation has been generalized to an arbitrary 
algebraic torus $T$ over $k$ by
Gonzalez-Aviles~\cite{Gonzalez} when the modulus $\grm$ is trivial, 
where the present formula may be viewed as the special case 
$T=\mathbb{G}_{m,k}$.
Furthermore, when $K/k$ is cyclic, we explain that 
Theorem~\ref{A.1} is essentially equivalent to 
Chevalley's ambiguous class number 
formula (the case $\calC=0$ in Theorem~\ref{CG}), 
thanks to the theorem on the Herbrand quotient of global $S$-units. 
The argument of the proof of Theorem~\ref{A.1} is slightly different from 
Lang's exposition \cite[Chapter XIII, Section 4]{lang:cyc}.

We keep the notation of Section \ref{sec:CG.1}. 
Recall that we write $\Cl^\grm_{K,S}$ for $\Cl^{\grm_K}_{K,S_K}$.
Let  $E^{\grm}_{K, S}$ be the intersection of the group of $S$-units
of $K$ with $K^\grm$. We have the exact sequence 
\begin{equation} 
\label{eq:A.1}
1 \rightarrow E^{\grm}_{K, S} \rightarrow K^\grm \xrightarrow {i_S}
P^{\grm,S}_{K}  \rightarrow 1.
\end{equation}
To state the main result, we need to separate the infinite part $\grm_\infty$ of the modulus $\grm=\grm_f\grm_\infty$. Write
 \begin{equation}\label{eq: divides modulus}
\grm_\infty=\grm^{r}_\infty\grm^{c}_\infty   \quad \text{ and }  \quad  \grm_{r}:=\grm_f \grm^r_\infty, \end{equation} where $\grm^{r}_\infty$ is the product of the real places $v$ dividing $\grm_\infty$ such that $v$ is unramified in $K$ (i.e. $v$ stays real in $K$) and $\grm^c_\infty$ is the product of those $v$ such that $v$ becomes complex in $K$. Note that 
\begin{equation}\label{eq: G-inv-modulus}
(K^\grm)^G= K^\grm\cap k^\times=k^{\grm_r} \quad   \text{ whence } \quad  (E_{K,S}^\grm)^G=E_{k,S}^{\grm_r}.\end{equation}

\begin{thm}\label{A.1}
Let $K/k$ be a finite Galois extension of global fields with Galois
group $G$.
Let $\grm$ be a modulus of $k$ and $\grm_r$ be as in \eqref{eq: divides modulus}. Let $S\supset V_{k,\infty}$ 
be a finite non-empty set of places of $k$ such that $S\cap
S(\grm_f)=\emptyset$.  
Then  
\[  \frac{|(\Cl^{\grm}_{K,S})^G|}{|\Cl^{\grm_r}_{k,S}|}
  = \frac{ |H^2(G, E^\grm_{K,S})|}{|H^1(G,E^\grm_{K,S})|}   
  \frac{\prod_{v\notin S\cup S(\grm)}e_v\cdot  |H^1(G,K^\grm)|}
  {|\mathrm{Im} \{H^2(G,
  E^\grm_{K,S}) \rightarrow H^2(G, K^\grm)\}|}.  \]
\end{thm}

When the support  $S(\grm)$ of $\grm$ is empty, 
the term $H^1(G,K^\grm)=H^1(G,K^\times)$ is trivial 
by Hilbert's Theorem 90.
For the general case, we have the following formula. 

\begin{proposition}\label{A.2}
Let the notation and the assumptions be the same as in
Theorem~\ref{A.1}. 
Then \[H^1(G, K^\grm)\cong \prod_{v\in  S(\grm_f)}
H^1(G, \prod_{w\mid v} 1+\grm \co_w). \] 
Furthermore, if the extension $K/k$ is cyclic, then \[|H^1(G, K^\grm)|= \prod_{v\in
S(\grm_f)} [1+\grm\co_v: N(\prod\limits_{w\mid 
v}(1+\grm\co_w))].\] 
\end{proposition}
 
The following proposition relates $|\Cl^\grm_{k,S}|$ with $|\Cl^{\grm_r}_{k,S}|$. One can find a proof using the language of ideals in \cite[Chapter~V, Theorem 1.7]{milne:cft} for example; we shall present an adelic proof for the sake of completeness.
\begin{proposition}\label{A.3}
	Let the notation and the assumptions be the same as in
	Theorem~\ref{A.1}. Then 
	
	$(1)$ $|\Cl^\grm_{k,S}|=|\Cl_{k,S}|\cdot [E_{k,S}:E^\grm_{k,S}]^{-1} \cdot 2^{|S(\grm_\infty)|}\cdot \left|  (\co_{k,S}/\grm_f\co_{k,S})^\times \right|$;
	
	$(2)$
	$|\Cl^\grm_{k,S}|=|\Cl^{\grm_r}_{k,S}|\cdot [E^{\grm_r}_{k,S}:E^\grm_{k,S}]^{-1} \cdot 2^{|S(\grm^c_\infty)|}. $
\end{proposition}

\begin{lemma}\label{A.lem}
We have

$(1)$ $H^1(G,I^S_K)=0$ and  $({I^S_K})^G/{I^S_k} \cong  \oplus_{v\notin S}\Z/{e_v\Z}$;

$(2)$  $(P^{\grm,S}_K)^G/P^{\grm_r,S}_k \cong \Ker \varphi$, where $\varphi$ is the natural map $H^1(G,E^\grm_{K,S}) \rightarrow H^1(G,K^\grm)$; 

$(3)$ $|H^1(G,P^{\grm,S}_K)|=|H^1(G,K^\grm)|\cdot |\Ker \psi|\cdot
|\mathrm{Im}\text{ } \varphi|^{-1}$, where $\psi$ is the map
$H^2(G,E^\grm_{K,S}) \rightarrow H^2(G,K^\grm)$.  
\end{lemma}
\begin{proof}
(1) For each place $v$ of $k$, let $G_v$ be a decomposition group of $v$,
which is uniquely determined up to conjugate.  
Since $G$ acts transitively on the set of places of $K$ above $v$,  
$I^S_K\cong \oplus_{v\notin S} \oplus_{w\mid v} \Z w =
\oplus_{v\notin S} \mathrm{Ind}^{G}_{G_v}\Z $. By  Shapiro's Lemma,
$H^1(G,I^S_K) =\oplus_{v\notin S}H^1(G_v,\Z)=0$.

For each finite place $v$ of $k$, let $\grp_v$ be the corresponding 
prime ideal of $\calO_k$, and $\gra_v:=\prod_{\grP| \grp_v} \grP$ the
prime ideal of $\calO_K$ such that $\gra_v^{e_v}=\grp_v \calO_K$. It is
clear that $(I^S_K)^G$ and $I_k^S$ are free abelian groups generated by
$\gra_v$ and $\grp_v$ for all $v\not\in S$, respectively. Thus,
$(I^S_K)^G/{I^S_k}\cong \oplus _{v\notin S}\Z/{e_v\Z}$.

(2) Taking Galois cohomology of the exact sequence \eqref{eq:A.1},
    we get the long exact sequence 
\[ 1 \rightarrow (E^\grm_{K,S})^G \rightarrow (K^\grm)^G \rightarrow 
(P^{\grm,S}_{K})^G \rightarrow H^1(G, E^\grm_{K,S}) \xrightarrow []{\varphi}    H^1(G,K^\grm). \]
We have $(K^\grm)^G=k^{\grm_r}$.
It follows that  $(P^{\grm,S}_K)^G/P^{\grm_r,S}_k \cong \Ker \varphi$.

(3) Taking Galois cohomology of the exact sequence \eqref{eq:A.1},
    we get the long exact sequence 
\[ H^1(G, E^\grm_{K,S}) \xrightarrow []{\varphi} 
    H^1(G,K^\grm) \rightarrow H^1(G,K^\grm) 
   \rightarrow H^1(G,P^{\grm,S}_{K}) \rightarrow H^2(G,
   E^\grm_{K,S}) \xrightarrow{\psi} H^2(G,K^\grm)
   \]
and an exact sequence
\[ 0 \rightarrow {\rm Im} \varphi \rightarrow H^1(G,K^\grm) 
   \rightarrow H^1(G,P^{\grm,S}_{K}) \rightarrow \Ker {\psi}
   \rightarrow 0.  \]
From this the statement (3) follows.
\end{proof}

\begin{proof}[Proof of Theorem~\ref{A.1}]
Consider the exact sequence  of $G$-modules 
\[0\rightarrow P^{\grm,S}_K \rightarrow I^{S(\grm)\cup S}_K\rightarrow
\Cl^\grm_{K,S}\rightarrow 0.\] 
Taking  Galois cohomology, we obtain the following exact commutative diagram
\[ \begin{xymatrix}{
         0  \ar@{->}[r] & P^{\grm_r,S}_k \ar@{->}[r] \ar@{->}[d]
          & I^{S(\grm)\cup S}_k\ar@{->}[r]\ar@{->}[d] & \Cl^{\grm_r}_{k,S}
          \ar@{->}[r] \ar@{->}^-{j}[d] & 0  & \\ 
          0 \ar@{->}[r]  & (P^{\grm,S}_K)^G \ar@{->}[r]  &
     (I^{S(\grm)\cup S}_K)^G \ar@{->}[r]  & (\Cl^\grm_{K,S})^G\ar@{->}[r]  &
              H^1(G,P^{\grm,S}_K) \ar@{->}[r]  &0     \\    
        }\end{xymatrix} \]
We remark that the map $j$ is not injective in general. 
The snake lemma gives the exact sequence
\[
0\rightarrow \Ker j \rightarrow
(P^{\grm,S}_K)^G/P^{\grm_r,S}_k \rightarrow  (I^{S(\grm)\cup S}_K)^G /{I^{S(\grm)\cup S}_k}
\rightarrow (\Cl^\grm_{K,S})^G/{\mathrm{Im}\text{ } j}\rightarrow H^1(G,P^{\grm,S}_K)
\rightarrow 0.  \]
By Lemma~\ref{A.lem}, we have
\[ \begin{split} 
|(\Cl^\grm_{K,S})^G/{\Cl^{\grm_r}_{k,S}}| &
=|(\Cl^\grm_{K,S})^G/{\mathrm{Im}\text{ } j}|\cdot |\Ker j|^{-1}  \\  
& = \frac{
 |H^1(G,K^\grm)|\cdot |\Ker \psi|\cdot |\mathrm{Im}\text{ } \varphi|^{-1}}{|\Ker
 \varphi|} \cdot 
\prod_{v\notin S(\grm)\cup S}e_v   \\ 
&= \frac{|H^1(G,K^\grm)|\cdot |\Ker \psi|} {|H^1(G,E^\grm_{K,S})|}
\cdot \prod_{v\notin S(\grm)\cup S}e_v  \\
& = \frac{|H^1(G,K^\grm)|\cdot |H^2(G,E^\grm_{K,S})|}
{|H^1(G,E^\grm_{K,S})| |\mathrm{Im}\text{ } \psi |} 
\cdot \prod_{v\notin
  S(\grm)\cup S}e_v. \quad \text{ } \end{split}  \]  
This completes the proof of Theorem~\ref{A.1}.
\end{proof}

\begin{proof}[Proof of Proposition~\ref{A.2}]
 The facts  $H^1(G, \A^\times_K)=H^1(G, K^\times)=0$ and
 $(\A^\times_K)^G=\A^\times_k$ will be used. 
Taking Galois cohomology of the short exact sequence 
\[ 1\rightarrow K^\grm\rightarrow K^\times\rightarrow
 K^\times/K^\grm\rightarrow 1,\] we get the exact sequence
\[1\rightarrow k^\times/k^{\grm_r} \rightarrow (K^\times/K^\grm)^G \rightarrow H^1(G,K^\grm)\rightarrow 1.\]
Recall that $\A^\grm_K=\{(a_w)_w \in \A^\times_K| a_w
\in 1+\grm \calO_w 
\text{ for } w\mid \grm \}$.  We have  $K^\grm=K^\times \cap
\A^\grm_K$ and  $\A^\times_K=K^\times \A^\grm_K$ by the weak approximation
theorem.  
This gives a natural isomorphism \[ K^\times/K^\grm \cong
\A^\times_K/\A^\grm_K.   \] 
Taking Galois cohomology of 
$1\rightarrow \A^\grm_K\rightarrow \A^\times_K
\rightarrow \A^\times_K/\A^\grm_K\rightarrow 1$, we get the exact
sequence 
\[ 1\rightarrow \A^\times_k/{\A^{\grm_r}_k} \rightarrow
(\A^\times_K/{\A^\grm_K})^G \rightarrow H^1(G, \A^\grm_K) \rightarrow
1.   \] 
This implies that 
\[ H^1(G, K^\grm) \cong H^1(G, \A^\grm_K).\] 
(Note
that this isomorphism can also be deduced from taking Galois cohomology of
the short exact sequence 
$1\rightarrow K^\grm \rightarrow \A^\grm_K\rightarrow
\A^\grm_K/K^\grm={\A^\times_K}/K^\times \rightarrow 1$ 
and using the facts that 
$({\A^\times_K}/K^\times)^G={\A^\times_k}/k^\times$ and
$H^1(G, {\A^\times_K}/K^\times)=0$.) 
Let $\A^{S(\grm)}_K$ be the subgroup of $\A^\times_K$ such that
\[\A^\times_K=\A^{S(\grm)}_K   \times   \prod_{w\mid
  \grm}K^\times_w \quad \text{ as a direct product}.\] 
We have  \[\A^\grm_K =  \A^{S(\grm)}_K \times
\prod_{w\mid \grm}1+\grm\co_w   \quad \text{      and    }  \quad
H^1(G,\A^{S(\grm)}_K )=0.\]  
It follows that
\[ H^1(G, \A^\grm_K) \cong \prod_{v\mid \grm} H^1(G, \prod_{w|v}1+\grm \co_w).   \] 
Observe that $H^1(G, \prod_{w\mid v}1+\grm \co_w) \cong H^1(G_v, 1+\grm
\co_{w_1})$ by Shapiro's Lemma,  where $w_1$ is a place of $K$ over $v$ and
$G_v=\Gal(K_{w_1}/{k_v})$. Also note that $H^1(G_v, 1+\grm\co_w)=0$ when $v$ is real.
This proves the first part of Proposition~\ref{A.2}. 

Assume that $K/k$ is cyclic. For a finite place $v$  of $k$, the
Herbrand quotient of the $\Gal(K_w/k_v)$-module $\co^\times_w$ is $1$; see \cite[Chapter IX, \S3, Lemma 4]{lang:ant}. Note that  $1+\grm\co_w$ has
finite index in $\co^\times_w$.   We then have \[|H^1(G_v, 1+\grm \co_w
)|=|H^2(G_v, 1+\grm \co_w)|=[1+\grm \co_v: N_{K_w/k_v}(1+\grm \co_w)
].\] 
This completes the proof of Proposition~\ref{A.2}.   
\end{proof} 

\begin{proof}[Proof of Proposition~\ref{A.3}]
	As in the proof of Theorem~\ref{CG}, we have
	\[  \Cl^\grm_{k,S} \cong \A^\times_k/k^\times U^\grm_{k,S},  \]
	where 
	\[U^{\grm}_{k,S}=\prod_{v\nmid \grm, v\not\in S} 
	\co^\times_v \prod_{v\mid \grm}
	1+\grm\co_v \prod_{v\in S,
		v\nmid \grm_\infty} k^\times_v.\] 
	It follows that \[   \frac{|\Cl^\grm_{k,S}|}{|\Cl_{k,S}|}=| k^\times U_{k,S}/k^\times U^{\grm}_{k,S}|. \]
	Consider the exact sequence 
	\begin{equation}\label{eq: exact}  1\rightarrow U_{k,S}\cap k^\times U^{\grm}_{k,S}/U^{\grm}_{k,S}\rightarrow    U_{k,S}/U^{\grm}_{k,S}   \rightarrow k^\times U_{k,S}/k^\times U^{\grm}_{k,S}\rightarrow 1.    \end{equation}
	Clearly the middle term has order 
	\[ 2^{|S(\grm_\infty)|}\cdot\prod_{v\mid \grm_f}|(\co^\times_v/{1+\grm\co_v})|= 2^{|S(\grm_\infty)|}\cdot\prod_{v\mid \grm_f}|(\co_v/\grm\co_v)^\times|=2^{|S(\grm_\infty)|}\cdot |  (\co_{k,S}/\grm_f\co_{k,S})^\times |.  \]
	The last equality follows from the Chinese remainder theorem. 
	
	Suppose $A,B$ and $C$ are subgroups of some abelian group (written multiplicatively) such that $C\subset A$. Then it is direct to check that the natural map $B\cap A \hookrightarrow A\cap BC$ induces an isomorphism
	\[  B\cap A/B\cap C \cong A\cap BC/C.  \]
	Applying this to $A=U_{k,S}, B=k^\times$ and $C=U^\grm_{k,S}$ shows that the first term of \eqref{eq: exact} is isomorphic to $E_{k,S}/{E^\grm_{k,S}}$. This proves formula (1).
	
	Formula (2) follows from applying (1) to the modulus $\grm$ and $\grm_r$ respectively. \end{proof}

\begin{remark}
Suppose that $K/k$ is cyclic. 
We can identify $H^2(G, M)$ with
the Tate cohomology $\wh{H}^{2}(G,M)\cong \wh{H}^{0}(G,M)$ by
periodicity for any $G$-module $M$.  
The theorem on the Herbrand quotient of global units says 
(see \cite[Chapter IX, \S4, Corollary 2]{lang:ant}) 
\[  \frac{|H^2(G,E_{K,S})|}{|H^1(G,E_{K,S})
    |}=\frac{\prod_{v\in S}|G_v|}{[K:k]}.  \]
Here $G_v$ is the decomposition group of $v$ and $E_{K,S}$ is the
 group of $S$-units of $K$.  Note that $E^\grm_{K,S}$  has finite index in
 $E_{K,S}$. So  they have the same Herbrand quotient.   Since 
 $H^2(G, E^\grm_{K,S})=E^{\grm_r}_{k,S}/N(E^\grm_{K,S})$
and $H^2(G,K^\grm)=k^{\grm_r}/N(K^\grm)$, we have 
\[|\mathrm{Im}\text{ }\psi
|=|E^{\grm_r}_{k,S}N(K^\grm)/N(K^\grm)|=[E^{\grm_r}_{k,S}:
E^{\grm_r}_{k,S}\cap N(K^\grm)]. \] 
Thus, by Theorem~\ref{A.1} and 
Proposition~\ref{A.2},   
we obtain the ambiguous class number formula for $S$-ray class groups  
\[   \frac{|(\Cl^\grm_{K,S})^G|}{|\Cl^{\grm_r}_{k,S}|} =
\frac{\prod\limits_{v\in S}e_v f_v \cdot \prod\limits_{v\in S(\grm_f)}
  [1+\grm\co_v: N(\prod\limits_{w\mid 
                v}(1+\grm\co_w))] \cdot \prod\limits_{v\notin
  S(\grm)\cup S}e_v    }{[K:k] [E^{\grm_r}_{k,S}:       E^{\grm_r}_{k,S}\cap
  N(K^\grm)]}.  \] 
To compare this formula to Theorem~\ref{CG}, we first note that $E^{\grm_r}_{k,S}\cap N(K^\grm)=E^{\grm}_{k,S}\cap N(K^\grm)$. Then by Proposition~\ref{A.3}(2), the above formula gives 
\[ \frac{|(\Cl^\grm_{K,S})^G|}{|\Cl^{\grm}_{k,S}|} = \frac{\prod\limits_{v\in S\setminus S(\grm_\infty)}e_v f_v \cdot \prod\limits_{v\in S(\grm_f)}
	[1+\grm\co_v: N(\prod\limits_{w\mid 
		v}(1+\grm\co_w))] \cdot \prod\limits_{v\notin
		S(\grm)\cup S}e_v    }{[K:k] [E^{\grm}_{k,S}:       E^{\grm}_{k,S}\cap
	N(K^\grm)]}. \]
This is the formula in Theorem~\ref{CG} when $\mathcal{C}=0$.
\end{remark}

\subsection*{Acknowledgments.} We thank Professor Georges Gras for his helpful comments and 
	Professor Christian Maire for informing us about his thesis \cite{maire:thesis}, which is related to our Section 4.
	Li is supported  by Anhui Initiative in Quantum Information Technologies
	(Grant No.~AHY150200), NSFC (Grant No.~11571328) and the Fundamental
	Research Funds for the Central Universities (No. WK0010000058). 
	Yu was partially supported by the grants 
	MoST 103-2918-I-001-009 and 104-2115-M-001-001-MY3. We are grateful to
	the anonymous referee for his/her helpful comments and encouragement to
	extend the previous results of Section 4 to $S$-ray class groups; 
	these improve the paper significantly.


\bigskip 

\begin{thebibliography}{10}
        
        \def\jams{{\it J. Amer. Math. Soc.}} 
        \def\invent{{\it Invent. Math.}} 
        \def\ann{{\it Ann. Math.}} 
        \def\ihes{{\it Inst. Hautes \'Etudes Sci. Publ. Math.}} 
        \def\ecole{{\it Ann. Sci. \'Ecole Norm. Sup.}}
        \def\ecole4{{\it Ann. Sci. \'Ecole Norm. Sup. (4)}}
        \def\mathann{{\it Math. Ann.}} 
        \def\duke{{\it Duke Math. J.}} 
        \def\jag{{\it J. Algebraic Geom.}} 
        \def\advmath{{\it Adv. Math.}}
        \def\compos{{\it Compositio Math.}} 
        \def\ajm{{\it Amer. J. Math.}}
        \def\grenoble{{\it Ann. Inst. Fourier (Grenoble)}}
        \def\crelle{{\it J. Reine Angew. Math.}}
        \def\mrl{{\it Math. Res. Lett.}}
        \def\imrn{{\it Int. Math. Res. Not.}}
        \def\acad{{\it Proc. Nat. Acad. Sci. USA}}
        \def\tams{{\it Trans. Amer. Math. Sci.}}
        \def\cras{{\it C. R. Acad. Sci. Paris S\'er. I Math.}} 
        \def\mathz{{\it Math. Z.}} 
        \def\cmh{{\it Comment. Math. Helv.}}
        \def\docmath{{\it Doc. Math. }}
        \def\asian{{\it Asian J. Math.}}
        \def\jussieu{{\it J. Inst. Math. Jussieu}}
        \def\plms{{\it Proc. London Math. Soc.}}
        
        \def\manmath{{\it Manuscripta Math.}} 
        \def\jnt{{\it J. Number Theory}} 
        \def\ijm{{\it Israel J. Math.}}
        \def\ja{{\it J. Algebra}} 
        \def\pams{{\it Proc. Amer. Math. Sci.}}
        \def\smfmemoir{{\it Bull. Soc. Math. France, Memoire}}
        \def\bsmf{{\it Bull. Soc. Math. France}}
        \def\sb{{\it S\'em. Bourbaki Exp.}}
        \def\jpaa{{\it J. Pure Appl. Algebra}}
        \def\jems{{\it J. Eur. Math. Soc. (JEMS)}}
        \def\jtokyo{{\it J. Fac. Sci. Univ. Tokyo}}
        \def\cjm{{\it Canad. J. Math.}}
        \def\jaums{{\it J. Australian Math. Soc.}}
        \def\pspm{{\it Proc. Symp. Pure. Math.}}
        \def\ast{{\it Ast\'eriques}}
        \def\pamq{{\it Pure Appl. Math. Q.}}
        \def\nagoya{{\it Nagoya Math. J.}}
        \def\forum{{\it Forum Math. }}
        \def\blms{{\it Bull. London Math. Soc.}}
        \def\aa{{\it Acta Arith.}}
        \def\bimas{{\it Bull. Inst. Math. Acad. Sin. (N.S.)}}
        \def\jrms{{\it J. Ramanujan Math. Soc.}}
        
        \def\tp{{to appear}}
        
 \newcommand{\princeton}[1]{Ann. Math. Studies #1, Princeton
                Univ. Press}
        
 \newcommand{\LNM}[1]{Lecture Notes in Math., vol. #1, Springer-Verlag}
 
 
 
        


\bibitem{artin-tate} E.~Artin and J.~Tate,  {\it Class field theory.}
AMS Chelsea Publishing,  2008. 



\bibitem{caputo-nuccio}L.~Caputo and F. A. E. Nuccio  Mortarino Majno
Di Capriglio, Class number formula for dihedral extensions. 
{\it Glasgow Math. J.}~{\bf 62} (2020), 323--353. 




\bibitem{chevalley} C.~Chevalley,  Sur la th\'eorie du corps de classes
dans les corps finis et les corps
locaux. 1934. \url{http://eudml.org/doc/192833}. 


\bibitem{Gonzalez} C. D.~Gonzalez-Aviles,  Chevalley's ambiguous
class number formula for an arbitrary torus. \mrl~{\bf 15} 
(2008), no. 6, 1149--1165.




\bibitem{gras:73} G. Gras, Sur les $\ell$-classes d'id{\'e}aux
dans les extensions cycliques  relatives de degr{\'e} premier
$\ell$. {\it Ann. Inst. Fourier (Grenoble)} \textbf{23} (1973), no. 3,
1--48. 


\bibitem{gras:94} G.~Gras, Classes g\'{e}n\'{e}ralis\'{e}es
invariantes.  {\it J. Math. Soc. Japan}~{\bf46} (1994), no. 3,  467--476.


\bibitem{gras:cft} G.~Gras,  {\it Class field theory. From theory to practice.}   Springer Monographs in Mathematics. Springer-Verlag, 2003.




\bibitem{gras:17} G.~Gras, Invariant generalized ideal
classes--structure theorems for $p$-class groups in
$p$-extensions. {\it Proc. Indian Acad. Sci. Math. Sci.}~\textbf{127}
(2017), no. 1, 1--34.      


\bibitem{lang:ant} S.~Lang, {\it Algebraic number theory.} Graduate
Texts in Mathematics~{\bf 110},  Springer-Verlag, 1986.


\bibitem{lang:cyc} S. Lang, {\it Cyclotomic fields I and II.} Graduate
Texts in Mathematics~{\bf 121}, Springer-Verlag, 1990. 


\bibitem{lemmermeyer:amb} F.~Lemmermeyer,  The ambiguous class
number formula revisited. \jrms~{\bf 28} (2013), no. 4,
415--421.




\bibitem{ljn:clgp} J. Li, Y. Ouyang,  Y. Xu and  S. Zhang,
$\ell$-Class groups of fields in  Kummer towers. arXiv 1905.04966v2,
2019. 



\bibitem{maire:thesis} C.~Maire,  T-S capitulation. \emph{Th{\'e}orie des
    nombres, Ann{\'e}es}  1994/95--1995/96, 33 pp., 
   Publ. Math. Fac. Sci. Besan\c{c}on, Univ. Franche-Comt{\'e},
    Besan\c{c}on, 1997. 


\bibitem{milne:cft} J.S. Milne,  {\it Class Field Theory.} (v4.02) (2013). Available at \url{http://www.jmilne.org/math/}. 



\bibitem{rosen:amb-funfield} M.~Rosen,  Ambiguous divisor classes in
function fields. \jnt~{\bf 9} (1977), no. 2 160--174. 



\bibitem{stevenhagen:redei}P.~Stevenhagen, Redei reciprocity,
  governing fields, and negative Pell. arXiv:1806.06250v2, 2020. 


\bibitem{yu:amb} C.-F. Yu, A remark on Chevalley's ambiguous class
  number formula. arXiv:1412.1458v2, 2016. 
 
        
\end{thebibliography}
\end{document}